\documentclass{amsart}
\usepackage{amssymb}
\usepackage{color}
\usepackage{graphicx}
\usepackage{amsmath,amscd}
\usepackage{hyperref}
\usepackage{tikz-cd}

\usepackage[all,cmtip]{xy}

\newcommand{\la}{\lambda}
\newcommand{\La}{\Lambda}

\newcommand{\p}{\phi}

\newcommand{\D}{\Delta}

\newcommand{\SO}{{\mathcal{O}}}

\newcommand{\Op}{{\mathcal{O}} p}

\newcommand{\BB}{\mathbb{B}}

\newcommand{\C}{\mathbb{C}}

\newcommand{\R}{\mathbb{R}}

\newcommand{\W}{\mathbb{W}}
\newcommand{\DD}{\mathbb{D}}

\renewcommand{\S}{\mathbb{S}}

\newcommand{\sse}{\subseteq}

\newcommand{\st}{\text{st}}
\newcommand{\ot}{\text{ot}}

\newcommand{\Cont}{\operatorname{Cont}}
\newcommand{\Symp}{\operatorname{Symp}}
\newcommand{\ob}{\operatorname{ob}}

\newcommand{\im}{\operatorname{im}}

\newcommand{\lr}{\longrightarrow}

\newcommand{\dd}{\partial}

\newtheorem{proposition}{Proposition}
\newtheorem{theorem}[proposition]{Theorem}

\newtheorem{lemma}[proposition]{Lemma}

\newtheorem{remark}[proposition]{Remark}

\begin{document}

\title{The Legendrian Whitney trick}

\date{}


\author{Roger Casals}
\address{University of California Davis, Dept. of Mathematics, Shields Avenue, Davis, CA 95616, USA}
\email{casals@math.ucdavis.edu}

\author{Dishant M. Pancholi}
\address{Chennai Mathematical Institute,
H1 SIPCOT IT Park, Kelambakkam,
Siruseri Pincode:603 103, TN, India.}
\email{dishant@cmi.ac.in}

\author{Francisco Presas}
\address{Instituto de Ciencias Matem\'aticas CSIC-UAM-UC3M-UCM,
C. Nicol\'as Cabrera, 13-15, 28049, Madrid, Spain}
\email{fpresas@icmat.es}

\begin{abstract}
In this article, we prove a Legendrian Whitney trick which allows for the removal of intersections between codimension-two contact submanifolds and Legendrian submanifolds, assuming such a smooth cancellation is possible. This technique is applied to show the existence h-principle for codimension-two contact embeddings with a prescribed contact structure.
\end{abstract}
\maketitle


\section{Introduction}\label{sec:intro}

The object of this article will be to show the existence of the Legendrian Whitney trick, which removes smoothly canceling intersections between codimension-2 contact submanifolds and Legendrian submanifolds.

The smooth Whitney trick is a method for removing points of intersection between two smooth submanifolds \cite{Kir,Whitney44}. This technique rests at the center of differential topology, with direct applications to embeddings problems \cite{Sha,Whitney44} and the h-cobordism theorem \cite{MilnorHCob,Smale}. Since its first appearance \cite{Whitney44}, the Whitney trick has been generalized to include multiple submanifolds \cite{MW}, intersections of positive dimension \cite{Go, Ha} and adapted for 4-dimensional manifolds \cite{Casson,Freedman}. At its core, the Whitney trick states that a smooth cancellation is possible if certain algebraic obstructions vanish, thus measuring the difference between algebraic topology and geometric topology. In the present article, we study the Legendrian Whitney trick, in the context of contact structures, further crystallizing the difference between smooth topology and contact topology.

Let $(\DD^{2n-1} , \xi_{st})$ be the standard contact structure in the smooth $(2n-1)$-dimensional disk, and $(B,\xi)$ an arbitrary contact structure in the $(2n+1)$-dimensional standard smooth disk
$$B=\{(x,y,z)\in\R^{n}\times\R^{n-1}\times\R^2:|x|^2+|y|^2+|z|^2\leq1\}.$$
The two smooth inclusions
$$\DD^{2n-1}=\{(x,y,z)\in\R^{n}\times\R^{n-1}\times\R^2:|x|^2+|y|^2\leq1, z=0\}\subseteq B,$$
$$\S^{n}=\{(x,y,z_1,z_2)\in\R^{n}\times\R^{n-1}\times\R\times\R:|x|^2+|z_1|^2=0.5,p=z_2=0\}\subseteq B,$$
will be referred to as the smoothly standard embeddings. The main result of the article reads as follows:

\begin{theorem}[Legendrian Whitney Trick] \label{thm:Legenvoid}
	Let $\phi: (\DD^{2n-1},\dd \DD^{2n-1};\xi_{st}) \lr (B,\dd B;\xi)$ be a proper isocontact embedding, $n\geq 2$, and $\la: \S^n \lr (B, \xi)$ a Legendrian embedding, such that $\phi,\la$ are smoothly standard. Then, there exists a compactly supported family of isocontact embeddings $\phi_t: (\DD^{2n-1} , \xi_{st}) \lr (B, \xi)$, $t\in[0,1]$, $\phi_0= \phi$, such that $\im(\phi_1)\cap\im(\la)=\emptyset$.
\end{theorem}

Theorem \ref{thm:Legenvoid} stands in contrast with many known obstructions \cite{BEM,CMP,El,Vo} for the removal of intersections,  via contact isotopies, between certain subsets in contact manifolds. In particular, pairs of overtwisted disks are known to admit a rigid behavior \cite{Vo} whereby there exists a smooth isotopy removing their intersections but no such contact isotopy exists. In fact, as shown in \cite[Section 8.2]{CMP}, the intersection theory between the overtwisted contact germ and certain Legendrian spheres is strictly richer than the smooth topology. The new insight in Theorem \ref{thm:Legenvoid} is that in higher dimensions the interaction between Legendrian submanifolds and contact submanifolds is much closer to smooth topology.

In the same vein as the smooth Whitney trick, Theorem \ref{thm:Legenvoid} should lead to a source of applications in higher-dimensional contact topology. In this article, we apply Theorem \ref{thm:Legenvoid} to resolve the existence problem for isocontact submanifolds:

\begin{theorem} \label{thm:main}
Let $(N, \xi_N)$ and $(M, \xi_M)$ be contact manifolds, with $\dim(M)= \dim(N)+2\geq5$, and $(f^0, F^0_s): (N, \xi_N) \lr (M, \xi_M)$ a formal isocontact embedding. Then there exists a family $(f^t, F^t_s)$ of formal isocontact embeddings of $N$ in $M$ such that $(f^1, F^1_s=df^1)$ is an isocontact embedding.
\end{theorem}

Theorem \ref{thm:main} shows that any codimension-2 smooth submanifold in a contact manifold $(M,\xi_M)$ can be smoothly isotoped to be a contact submanifold, with a chosen induced contact structure, as long as the obstructions from algebraic topology vanish. In pair with M. Gromov's h-principle for codimension-$k$ submanifolds, $k\geq4$, Theorem \ref{thm:main} resolves the existence h-principle for isocontact submanifolds in {\it any} formal class, i.e. the existence of codimension-2 contact embeddings with a prescribed contact structure in the domain.

The study of contact submanifolds in higher dimensions has recently seen new developments, including the articles \cite{CasalsEtnyre,CMP,EF,EL,HondaHuang19,Ol,PP}. Theorem \ref{thm:Legenvoid} focuses on the interplay between Legendrian submanifolds and contact submanifolds, the first result of its type. Theorem \ref{thm:main} addresses the {\it existence} of a contact submanifold in a given smooth class. Regarding {\it uniqueness}, the first author and J. Etnyre \cite[Theorem 1.1]{CasalsEtnyre} recently constructed the first examples of contact submanifolds in all dimensions which are smoothly isotopic but {\it not} contact isotopic. Amongst these new developments, the third author and S. Pandit \cite[Theorem 1]{PP} discuss which contact structures can be realized, and see also \cite[Theorem 4.12]{Ol} for an independent argument. In addition, the recent article \cite[Corollary 1.3.5]{HondaHuang19} gives a convex surface construction of contact representatives. Thus, this present work is part of a new collection of articles establishing the foundations for the study of contact submanifolds in higher dimensions. In particular, Theorem \ref{thm:main} should also follow by combining the results from the third author \cite{PP} with the higher-dimensional convex theory \cite{HondaHuang19}.

The present arguments for Theorem \ref{thm:Legenvoid} and \ref{thm:main} apply to higher dimensional contact manifolds, and crucially use the dimensional hypothesis. In the remaining case, that of transverse knots in contact 3-manifolds, Theorem \ref{thm:main} holds if we allow ourselves to change the self-linking number \cite[Section 3.3]{GeigesCont}, and thus the conclusion is only on the smooth type of the knot, rather than the formal isocontact type of the embedding. This situation is in line with the 3-dimensional Thurston-Bennequin inequality \cite[Section 4.6.5]{GeigesCont} and the difference between Legendrian knots and higher-dimensional Legendrian submanifolds \cite{CMP}.

{\bf Organization.} This article is organized as follows: Section \ref{sec:whitney} proves Theorem \ref{thm:Legenvoid} and Section \ref{sec:main} shows Theorem \ref{thm:main} as an application. Regarding notation, given a set $C$ we will denote by $\mathcal O p(C)$ an arbitrarily small open neighborhood of it.
%

{\bf Acknowledgements.} The authors are thankful to V.L. Ginzburg and A. del Pino for helpful conversations. R. Casals is also grateful to J.B. Etnyre for useful discussions in their collaboration, and a wonderful talk by J.V. Horn-Morris in the conference ``Geometric structures on 3 and 4 manifolds'', which particularly sparked his interest in the study of higher-dimensional contact embeddings. R.~Casals is supported by the NSF grant DMS-1841913 and a BBVA Research Fellowship. F. Presas is supported by the Spanish Research Projects SEV-2015-0554, MTM2016-79400-P, and MTM2015-72876-EXP. D. Pancholi and F. Presas want to thank ICTP (Trieste, Italy) for the support and the help that they have provided, through their visitors program sponsored by the Simons Foundation, during the two visits that they did to the Center in order to develop their work in this article.\hfill$\Box$
%


\section{Legendrian Whitney trick.} \label{sec:whitney}

The classical smooth Whitney trick \cite{Kir,Whitney44} starts with two oriented submanifolds $S_0,S_1\sse M$ inside a simply connected orientable manifold $M$, which intersect at finitely many (signed) points. The main assumption is the existence of two oppositely oriented intersection points $p_0,p_1\in S_0\cap S_1$. The Whitney trick then consists of canceling these two intersection points by placing them as boundaries of two embedded curves $\psi_0:[0,1]  \lr S_0$ and $\psi_1:[0,1] \lr S_1$. The crucial part of the argument relies on the existence of an embedded smooth disk $\psi:[0,1] \times[0,1]\lr M$ which intersects $S_0$ and $S_1$ exactly in the images of $\psi_i$, $i=0,1$, only along its boundary, and satisfies $\psi(t,0)=p_0$ and $\psi(t,1)=p_1$. This disk, known as a Whitney disk, is then used to construct a smooth flow that pushes $\psi_0$ to $\psi_t=\psi(t,-)$ for time $t$ and is supported in an arbitrarily small neighborhood of the disk. The image of $S_0$ through the flow at time $t=1+ \delta$, for $\delta\in\R^+$ arbitrarily small, becomes displaced from $S_1$ at the points $p_0,p_1\in S_1$, and no new intersections are introduced.

The Legendrian Whitney trick generalizes the smooth Whitney trick to submanifolds whose generic intersection has dimension strictly larger than zero. For instance, a Legendrian sphere in $(\S^{2n+1},\xi_\st)$ generically intersects a codimension-two contact submanifold in a smooth submanifold of dimension $(n-2)$. In this context, the underlying smooth arguments in the contact setting are in line with A. Haefliger's proof of unknotting in high codimensions \cite{Ha}, and see also A. Shapiro's theory of deformation cells \cite[Section 5]{Sha}.

Let $(M,\xi)$ be a $(2n+1)$-dimensional contact manifold, the intersection of a Legendrian sphere $S\sse(M,\xi)$ and a codimension-two contact submanifold $(\mathbb{D},\xi|_{\mathbb{D}})\sse(M,\xi)$  is a $(n-2)$-dimensional isotropic submanifold $\Sigma\sse \DD\cap S$. For the Legendrian Whitney trick, the smooth topology of $\Sigma$ might be non-trivial, and thus not necessarily bound a smooth disk. Instead, the role of the Whitney disk is played by a Legendrian embedding $\psi:W\times[0,1]\lr(M,\xi)$ restricting to two isotropic embeddings $\psi_0:W\times\{0\}\lr S$, $\psi_1:W\times\{1\}\lr \DD$ such that $\psi_0(\dd W\times\{0\})=\Sigma=\psi_1(\dd W\times\{1\})$. The image of the Legendrian embedding $\psi$ is referred to as a Legendrian Whitney bridge, and its interior will lie on the complement $M\setminus(S\cup\DD)$. It is this Legendrian Whitney bridge that allows us to construct a compactly supported contact isotopy sliding $S$ along $W$ and remove the intersection $\Sigma\sse S\cap\DD$ without creating a new one.

\subsection{Construction of the Legendrian Whitney Bridge}\label{ssec:LegWhitneyBridge}

First, let us work under the hypothesis of Theorem \ref{thm:Legenvoid}, with the contact submanifold $(\DD,\xi)\sse(B,\xi)$ being a standardly embedded contact disk and $S\sse(B,\xi)$ a Legendrian sphere which is also standardly embedded. By genericity, we assume the intersection $\Sigma=\DD\cap S$ between $\DD$ and $S$ is smoothly transverse. Since $S$ is compact and $\DD$ is smoothly standard, there exists a $(n-1)$-dimensional compact manifold $W$ and an embedding $W \lr S$ with boundary $\partial W = \Sigma$. Indeed, the submanifold $W$ can be constructed by fixing a $2n$-dimensional smoothly standard disk $\hat \DD^{2n}$, satisfying $\DD \sse \hat \DD \sse B$, and intersecting the closure of one of the two connected components of $\hat{\DD}^{2n} \setminus\DD$ with the Legendrian sphere $S$.

Let $C_\Sigma=\Sigma\times(-\varepsilon,\varepsilon]$ be a collar neighborhood of the boundary $\Sigma\sse W$, and denote by $\W$ the quotient of $W\times[0,1]$ by the relation $(w_1,t)\sim (w_2,s)$ if and only if $w_1,w_2\in\dd W$, where $w_1,w_2\in W$, $s,t\in[0,1]$, i.e. the quotient $e:W\times[0,1]\lr\W$ collapses the subset $\dd W\times[0,1]$ to a single copy of the boundary $\dd W$.

\begin{proposition}[Legendrian Whitney Bridge]\label{lem:Whitney_trick}
	Let $(\DD^{2n-1},\xi_{st})\sse(B,\xi)$ be a properly embedded contact disk, $S\sse(B, \xi)$ a Legendrian submanifold, and both inclusions smoothly standard. For any sufficiently small $\varepsilon\in\R^+$, there exist a compact manifold $W$, a Legendrian embedding $F: \W \lr B$ and a smooth quotient map $e: W \times [0, 1] \lr \W$ such that:
	\begin{enumerate}
		\item[1.] $F$ is transverse to $S$ and $\DD$, $F^{-1}(S)= e(W \times \{ 0 \})$ and $F^{-1}(\DD) = e(W \times \{ 1 \})$,
		
		\item[2.] $(F \circ e)(\partial W \times [0, 1])= \Sigma = S \cap\DD$, $(F\circ e)(\partial W\times[0,1]) = \Sigma$ and $e|_{\mathring W \times [0, 1]}$ is a smooth embedding,
		
		\item[3.] There exists a Legendrian ribbon of $\Sigma\sse B$ diffeomorphic to $\Sigma \times \mathbb D^2(3\varepsilon) \subset B$ such that the restriction $(F \circ e)|_{C_\Sigma \times [0,1]}$ is mapped to  $\Sigma \times \mathbb D^2(3\varepsilon)$ and there exist coordinates $(q, s, t)\in C_\Sigma \times [0,1]$ such that
		$$(F\circ e)(q,s,t)=(q,(s\cos(t\pi/2), s \sin(t\pi/2))).$$\hfill$\Box$
	\end{enumerate}
\end{proposition}

Property 3 in Proposition \ref{lem:Whitney_trick} is depicted schematically in Figure \ref{fig:corner}.

\begin{figure}[ht]
	\includegraphics[scale=0.45]{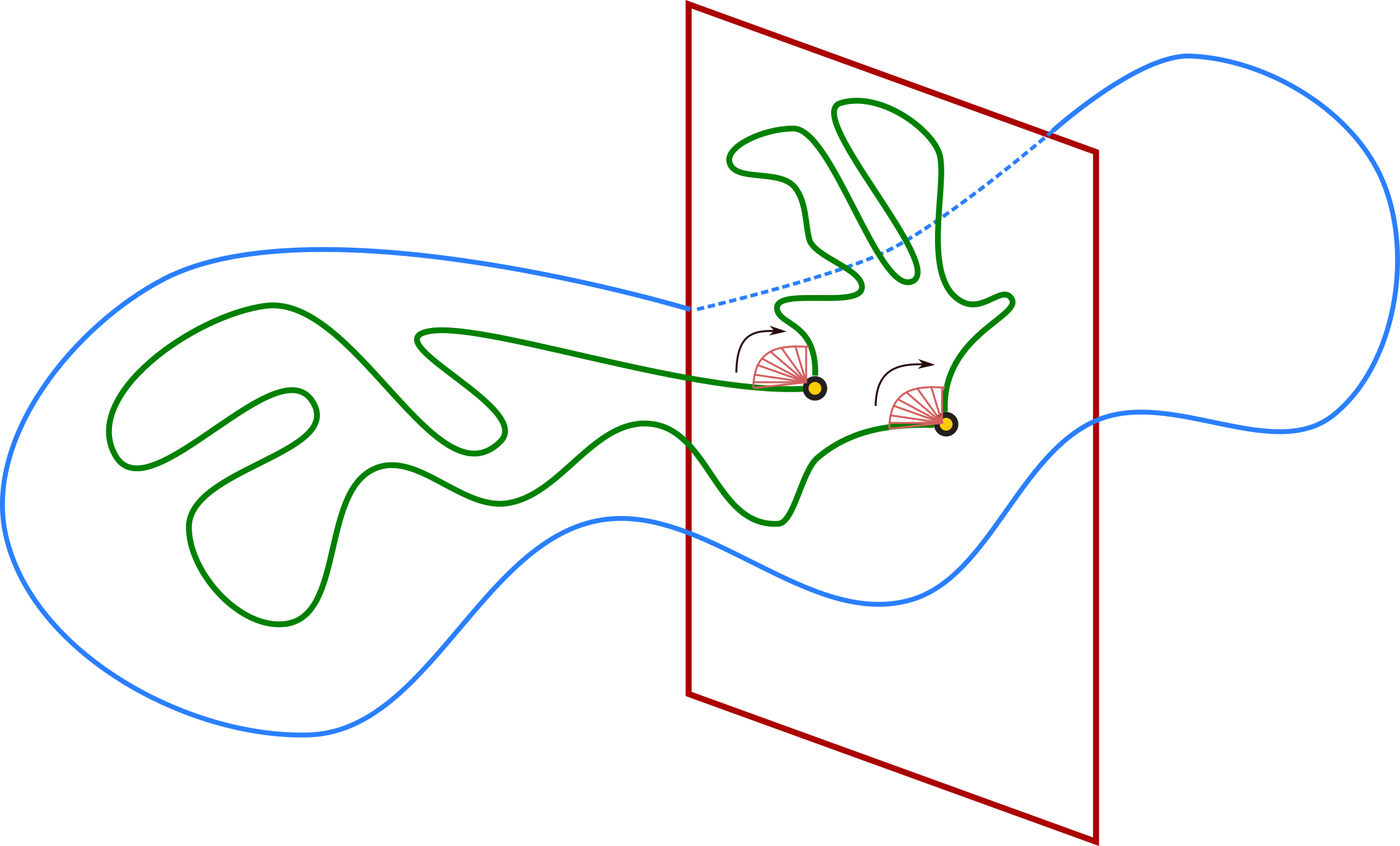} 
	\caption{The Legendrian neighborhood $\Sigma \times \mathbb D^2$ of the intersection $\Sigma=S\cap\DD$ and the image of the Legendrian Whitney bridge, where $W\times [0,1]$  is  mapped into the third quadrant.} \label{fig:corner}
\end{figure}

\begin{remark}
	The map $\psi = F \circ e:W\times[0,1]\lr B$  will be constructed as a $1$-parameter family $\{\psi_t\}_{t\in[0,1]}$ of isotropic embeddings of $W$, such that $\psi_0=(F\circ e)|_{W\times\{0\}}\sse S$ and $\psi_1=(F\circ e)|_{W\times\{1\}}\sse\DD$. Nevertheless, we first find a family of smooth embeddings $\{\phi_t\}_{t\in[0,1]}$ and then construct the required isotropic family $\{\psi_t\}_{t\in[0,1]}$. Hence, for the proof of Proposition \ref{lem:Whitney_trick},  $\phi_t$ are smooth morphisms, and $\psi_t$ are Legendrian morphisms.\hfill$\Box$
\end{remark}

Regarding notation, a normal bundle of a submanifold $A\sse B$ is denoted by $\nu_A^B$. In the case $\nu_A^B$ is an oriented real line bundle, $\nu_A^B$ will also denote a choice of no-where vanishing vector field trivializing the bundle.

\begin{subsection}{Proof of Proposition \ref{lem:Whitney_trick}} Let us declare $\phi_0 = \psi_0$ to be the inclusion $W\sse S$. Choose a compatible almost complex structure $J_{\xi}$ for the symplectic bundle $\xi=(\ker \alpha,d \alpha)$ which extends a complex structure compatible with $(\DD,\xi|_\DD)$. First, we proceed with the construction of the Legendrian embedding yielding $\W\sse B$ near the intersection locus $\Sigma$. For that, we choose framings as follows.
	
	Let us trivialize the rank-two (trivial) normal bundle $\nu_\Sigma^ S$ over $\Sigma$ such that
	$$\nu_\Sigma^S= \nu_\Sigma^{\phi_0(W)} \oplus \nu_{\phi_0(W)}^{ S}= \langle v_1 \rangle \oplus \langle w_1 \rangle,$$
	where $v_1,w_1$ are fixed but arbitrary non-vanishing sections. Since this rank-two bundle $\langle v_1, w_1 \rangle$ is transverse to $T\DD|_{\Sigma}$, there is a unique vector field $v_0$, up to positive scalar function, such that $\langle v_0 \rangle = \langle \{v_1, w_1, J_{\xi} w_1\} \rangle \cap T\mathbb D_{|\Sigma}$. In short, $v_0,v_1\in\Gamma(T\Sigma)$ are vector fields such that:
	\begin{itemize}
		\item[-] $v_0$ is tangent to $\DD$ and transverse to $\phi_0(W)$, $v_1$ is tangent to $\phi_0(W)$ and transverse to $\DD$,
		\item[-] The subbundle $T\Sigma \oplus \langle v_0, v_1 \rangle\sse \xi|_{\Sigma}$ is a Lagrangian space.
	\end{itemize}
	
	Indeed, the second property follows from the fact that $T\Sigma \oplus \langle v_1, w_1 \rangle$ is a Lagrangian subspace and we have replaced $w_1$ by a linear combination of the type $v_0= \lambda_0 w_1 + \lambda_1Jw_1 + \lambda_2 v_1$, and such a vector paired with the isotropic subspace $T\Sigma \oplus \langle v_1 \rangle $ evaluates to zero for the symplectic structure on $\xi_{\Sigma}$.
	
	For the smooth embedding $\phi_1$, we directly invoke the Whitney Embedding Theorem \cite{Whitney36,Whitney44}. This constructs a smooth embedding $\phi_1:W\lr\DD$ such that $\phi_1(\partial W)=\Sigma$, such that the vector field $v_0$ is tangent to $\phi_1(W)$ and $v_1$ is transverse to $\phi_1(W)$. In order to adequately extend the vector fields $v_0,v_1$ in a neighborhood of $\Sigma$, let us use the following:
	
	\begin{lemma}\label{lem:local_exten}
		Let $e:\Sigma\lr (M^{2n+1}, \xi)$ be an isotropic embedding, $k<n$, and $\{ \tau_{k+1}, \ldots, \tau_n\}$ a partial framing for $e^* \xi$ such that $T\Sigma \oplus \langle \{ e_* \tau_{k+1}, \ldots, e_* \tau_n \} \rangle \subset \xi_{e(S)}$ is a Lagrangian subspace. Then there exists a Legendrian embedding $\tilde{e}: \Sigma \times D^{n-k}(1) \lr (M, \xi)$ such that $\tilde{e}_{|\Sigma \times \{ 0 \}}=e$ and $(\tilde{e})_* \frac{\partial}{\partial x_i}|_{\Sigma \times \{ 0 \}}= \tau_{k+i}$, $1\leq i\leq n-k$.
	\end{lemma}
	
	Lemma \ref{lem:local_exten} is proven right after we conclude the present proof. Now, we apply the Lemma~\ref{lem:local_exten} to the inclusion $\Sigma=S\cap \DD$ and the partial framing $\{\tau_{n-1},\tau_n\}=\{v_0,v_1\}$, where the vector fields $v_0,v_1$ are extended preserving their containment in $\xi$ and the extension of $v_1$ being transverse to $\phi_1(W)$. This extension yields a family of maps $\{\psi_t\}=\{\phi_t\}$, $t\in[0,1]$, in the collar neighborhood $C_{\Sigma} \times [0,1]$ that conform to Condition 3 in the statement of Proposition \ref{lem:Whitney_trick}. Let us now work towards extending $\{\phi_t\}_{[0,1]}$ from the collar $C_{\Sigma} \times [0,1]$ to the entire $B$.
	
	First, let us claim that for $\varepsilon\in\R^+$ small enough, we can construct smooth push-offs of the two embeddings $\phi_0(W)$, $\phi_1(W)$ providing a family of embeddings $\{\phi_t\}$, $t\in[0,\varepsilon) \cup (1-\varepsilon, 1]$ which are each transverse to $\mathbb D$ and $S$, and intersect $\mathbb{D}$ and $\mathcal{S}$ at $\Sigma$. In order to prove this, it is enough to find two non-vanishing vector fields $\tau_0,\tau_1$ within $\xi$ which are defined over $\phi_0(W)$ and $\phi_1(W)$, respectively, are normal to them and extend $v_0$ and $v_1$, i.e. $(\tau_0)_{|\Sigma}=v_0$ and $(\tau_1)_{|\Sigma}=v_1$. For that, let $\D$ be the connected component of $\hat\DD \setminus D$ which defines $W$, and note that $\langle v_1 \rangle = \nu_\Sigma^{\phi_0(W) } \simeq (\nu_{\mathbb D}^{\Delta})|_{\Sigma}$. Then we choose $\tau_1$ along $\phi_1(W)$ such that $\langle \{ \tau_1 \} \rangle=\nu_{\mathbb D}^{\Delta}|_\{\phi_1(W)\}$. The vector field $\tau_0$ is constructed similarly. Indeed, the vector field $J_{\xi}\cdot w_1$ is defined along $\phi_0(W)$ and the vector field $v_0$ that is defined over the collar $C_\Sigma$. Since $\langle v_1, w_1 \rangle$ is transverse to $\mathbb D$, we have that $\lambda_1>0$ and thus the linear interpolation between $J_{\xi}\cdot w_1$ and $v_0$ is never zero and transverse to $\phi_0(W)$. Hence, we can choose a cut-off function $h:W \lr [0,1]$ compactly supported on the collar $C_{\Sigma}$, with $h_{|\Sigma}=1$ and declare $\tau_0= h v_0 +(1-h) J_{\xi} w_1$.

	In this stage of the argument, the Legendrian Whitney bridge $W\times[0,1]$, and its embedding $(F\circ e)$ are (smoothly) defined by $\phi$ in the region $A = W\times ([0,\varepsilon) \cup (1-\varepsilon, 1]) \cup C_{\Sigma} \times[0,1])$. Its extension can be constructed as follows: choose an arbitrary smooth map $\tilde \phi: W \times [0,1] \lr B$, which coincides with $\phi$ along $A$, and use Thom's Transversality Theorem \cite{EliashMisch} to obtain a $C^{\infty}$-small perturbation, relative to a neighborhood of $A$, which does not intersect $S$. This is possible since $\dim (W \times [0,1]) + \dim S = 2n < 2n+1 = \dim B$. Then, the Whitney Embedding Theorem applies to perturb the map $\tilde \phi$, relative to $A$, into a smooth embedding $\phi:W\times [0,1]\lr B$. The image of the restriction of this map $\phi$ to the open subset $\mathring W  \times (0,1)$ avoids both $S$ and $\DD$. At this stage of the proof, we have constructed a smooth Whitney bridge according to the statement of Proposition \ref{lem:Whitney_trick}. The remainder of the proof consists in deforming $\phi$ to a Legendrian embedding as required.
	
	The Legendrian embedding will be obtained via the $h$-principle on Legendrian immersions \cite[Section 16.1]{EliashMisch}, and thus it suffices to endow the smooth embedding $\phi:W\times[0,1]\lr B$ with the structure of a formal isotropic embedding $G:W \times [0,1] \lr \mbox{Mon}_{\R}(TW\oplus \R,\phi^*\xi)$. In fact, the bundle $\phi^* \xi$ decomposes as $V \oplus \C$ for a complex bundle $V$, satisfying that $\phi_0^* \xi= (\phi_0)_* TW \otimes_{\R} \C \oplus \langle \tau_0, J_{\xi} \tau_0 \rangle$ and $V|_{W \times \{ 1\}} \oplus \C$,  where $V_1 \simeq \xi|_{\mathbb{D}}$. Indeed, along the collar neighborhood  $C_\Sigma$, the complexification
	$(\phi_t)_* TW \otimes_{\R} \C$ is a choice for $V$, and the linear interpolation $(t \tau_1 + (1-t)\tau_0)_{|\Sigma}=t v_1 +(1-t) v_0$ is a section trivializing the rank-two symplectic orthogonal. Then obstruction theory \cite{HatcherBook} tells us that this trivial sub-bundle can be extended along $W\times[0,1]$: since $\xi$ is smoothly trivial, choose a smooth map $\hat{f}: W  \times [0,1] \lr \S^{2n-1}$ extending a fixed section on the boundary, which exists since $\dim W \times [0,1]=n$ and $\S^{2n-1}$ is $(2n-2)$--connected. In conclusion, we have the decomposition $\phi^* \xi= V \oplus \{ \hat{f}, J \hat{f} \}$ for the pull-back of the contact structure.
	
	Let us denote $V_t = V|_{\phi_t(W)}$ and 
	$G_t = G|_{W \times \{t\}}$, and let us require that the formal isotropic embedding $(\phi,G)$ satisfies the following properties:
	\begin{itemize}
		\item[(a)] $G_t(w)(\{0 \} \oplus \R)= \{ 0 \} \oplus \C \subset \xi_{(w,t)}$, for any $w\in W$,
		\item[(b)] $G(w,t) = d\phi(w,t)$, for $(w,t) \in (W \times [0, \varepsilon) \cup C_\Sigma  \times [0,1])$,
		\item[(c)] $G_t(w)(TW \oplus \{ 0 \}) \subset V_t$ and is a Lagrangian subspace of $V_t$, for any $w\in W$.
	\end{itemize}
	
	The first condition is determined by imposing $G(w,t)(\partial_t)= f(w,t)$, and the second condition merely formalizes that the embedding is already Legendrian, not just formally Legendrian, in certain regions. The third condition can be assumed since $W \times [0,1]$ deformation retracts to $(W \times [0, \varepsilon]) \cup (C_{\Sigma}\times [0,1])$. In conclusion, $(\phi, G)$ is a formal Legendrian immersion of $W\times [0,1]$ and $(\phi_1,G_1)$ is a formal Legendrian immersion on $\mathbb D$, since $V_1 \simeq \xi|_{\mathbb{D}}$. Now, the $C^0$--dense $h$--principle for Legendrian immersions and the genericity of the immersion yields a $C^0$--small perturbation of the pair $(\phi_1,G_1)$ into a new pair $(\psi_1,F_1'=d\psi_1)$. Then we push the smooth map $\psi_1$ along the flow of $\tau_1$ to define the maps $\psi_t$, for $t\in (1-\varepsilon,1]$, which are Legendrian embeddings. Finally, the same $C^0$--dense $h$--principle for Legendrians immersions relative to the domain for the pair $(\phi, G)$ produces the required Legendrian Whitney bridge. The embedding $F:\W\lr B$ defined in the statement is the map $\psi$ extended with the vector fields $\partial _s$ and $\partial_t$, defined in the two boundaries of  $W \times [0,1]$. This concludes the proof of Proposition \ref{lem:Whitney_trick}.\hfill$\Box$
\end{subsection}

\noindent {\bf Proof of Lemma \ref{lem:local_exten}}. Consider the contact structure $(T^*\Sigma \times \C^{n-k} \times \R,\xi_\st)$ considered as a $1$-jet space, and extend the map $e$ to an smooth embedding $\hat{e}: T^*\Sigma \times \C^{n-k} \times \R \lr \Op(\Sigma)$ such that $\hat{e}_* \xi_{std}$ coincides with $\xi$ over $e(\Sigma)$ and $\frac{\partial}{\partial x_i}$ is mapped to the corresponding framing vector. Since the associated conformal symplectic structures share a Lagrangian subspace, the linear interpolation between them is through conformal symplectic structures. By Moser's stability argument \cite{GeigesCont}, given that the contact structures coincide over $S$, there is a small neighborhood of $\Sigma$ that is contactomorphic to a small neighborhood of $\Sigma\sse T^*\Sigma \times \C^{n-k} \times \R$. Composing $\hat{e}$ with this contactomorphic yields the required Legendrian embedding.\hfill$\Box$

\subsection{The Legendrian Whitney Trick.}\label{ssec:LegWhitneyTrick} Let us prove Theorem \ref{thm:Legenvoid}. Apply Proposition \ref{lem:Whitney_trick} to construct a Legendrian Whitney bridge $(e,F,W)$ for the intersection $\Sigma=S\cap\DD$. Intuitively, the product structure $W\times[0,1]$ of the domain of the Legendrian Whitney bridge allows us to proceed in the same vein as in the smooth case \cite{Sha,Whitney44}. In the present case, the central caveat is that the isotopy that displaces the Legendrian sphere $S$ from the contact submanifold $\DD$ must also be a contact isotopy. The challenge thus relies in finding a compactly supported contact isotopy which pushes the Legendrian $S$ along the image of the Legendrian Whitney bridge $W\times[0,1]$ in such a manner that the last contactomorphism has removed the intersection $\Sigma$.

In the smooth case, this is achieved by integrating a compactly supported extension of the smooth vector field in the direction of $[0,1]$. The core difference between the smooth and the contact situations is the fact that smoothly cutting-off a contact vector field might not integrate to a contact isotopy. Hence, in a nutshell, the main contribution of this section is a careful analytical construction of a compactly supported contact vector field, equivalently a contact Hamiltonian, which produces a contact isotopy with the necessary properties. Let us introduce local coordinates in order to perform these computations.

\subsubsection{Coordinate System.} The image of the Legendrian Whitney Bridge $W\times[0,1]$ is the Legendrian submanifold $\W=(W\times[0,1])/\sim$ and thus an open neighborhood of the inclusion $\W\sse (B^{2n+1},\xi)$ can be identified with $J^1(\W,\xi_{st})$. Let $q_n\in[0,1]$ be the coordinate in the second factor of $W\times[0,1]$, and $q_{n-1}\in (-\varepsilon,\varepsilon)$ the local coordinate in $\Op(\dd W)$ defined by distance to $\dd W$. Thus, given the nature of the quotient $W\times[0,1]\lr\W$, the pair $(q_{n-1},q_n)$ can be understood as smooth polar coordinates in $\Op(\dd W)\sse\W$, parametrizing the 2-dimensional normal slice to $\dd W$, where $q_{n-1}$ is the radius and $q_n$ its associated angle.

The intersection $W_0=\W\cap S$ of the Whitney bridge with the Legendrian sphere is the image of $W\times\{0\}$, and the intersection $W_1=\W\cap\DD$ of the Whitney bridge with the contact submanifold is the image of $W\times\{1\}$. Far from $\Sigma$, i.e. $q_{n-1}\neq0$, the former can be given the local\footnote{The coordinates $(q_1,\ldots,q_{n-1})$ are only local coordinates in $W$.} equation $\{(q_1,\ldots,q_{n-1},q_n)\in \W: q_n=0\}$, whereas the latter is defined by $\{(q_1,\ldots,q_{n-1},q_n)\in \W: q_n=1\}$. Along $\Sigma$, given by the vanishing of $q_{n-1}$, the quotient $W\times[0,1]\lr\W$ forces $q_n$ to be ill-defined, and we interpret any restriction on $q_n$ being satisfied. Thus, the above equation can be understood as meaningful all along $\W$ and we have $W_0\cap W_1=\Sigma$. This establishes a coordinate system along $\W$, which we can now extend to its neighborhood $J^1(\W,\xi_{st})$. Indeed, the chosen extension is given by the trivial Legendrian inclusion $\W\sse J^1(\W,\xi_{st})\cong (T^*\W\times\R_\tau,\la_\W-d\tau)$ given by the zero section in $T^*\W$ at $\tau=0$.

In the model $T^*\W\times\R$, we write $(q,p)\in T^*\W$ for points in the cotangent bundle such that the canonical projection $\pi:T^*\W\lr\W$ is given by $\pi(q,p)=q$. The contact submanifold $\DD$ is cut-out by the equation
$$\DD=\{(q,p,\tau)\in T^*\W\times\R:q_n=1,p_n=0\},$$
where $p_n$ is the conjugate coordinate of $q_n$. The isotropy and dimension count for $\W,S$ and $\dd W$ forces their conjugate momenta to vanish and thus we can write $$S=\{(q,p,\tau)\in T^*\W\times\R:p_1=p_2=\ldots=p_{n-1}=0,\tau=q_n=0\}.$$
Again, these coordinates are understood in the quotient $\W$, and thus the intersection $\Sigma=S\cap\DD$ is non-empty, cut out by the equation $\Sigma=\{(q,p,\tau)\in T^*\W\times\R:p_1=p_2=\ldots=p_{n-1}=p_n=0,\tau=q_{n-1}=q_n=0\}$, since the vanishing of the radial coordinate $q_{n-1}$ allows for both equations $\{q_n=1\}$ and $\{q_n=0\}$ to be satisfied by the angular coordinate, as in the usual context of smooth polar coordinates. Finally, let $(x,y)\in D^2$ be Cartesian coordinates for the disk spanned by these polar coordinates $(q_{n-1},q_n)$, and denote their conjugate momenta by $(p_x,p_y)$. Note that, in these Cartesian coordinates the intersection of the contact submanifold $\DD$ with the Legendrian Whitney bridge reads
$$\DD \cap \SO p(\W) =\{(q,p,\tau)\in T^*\W\times\R:x=0,p_x=0\}.$$

\begin{figure}[ht]
	\includegraphics[scale=0.5]{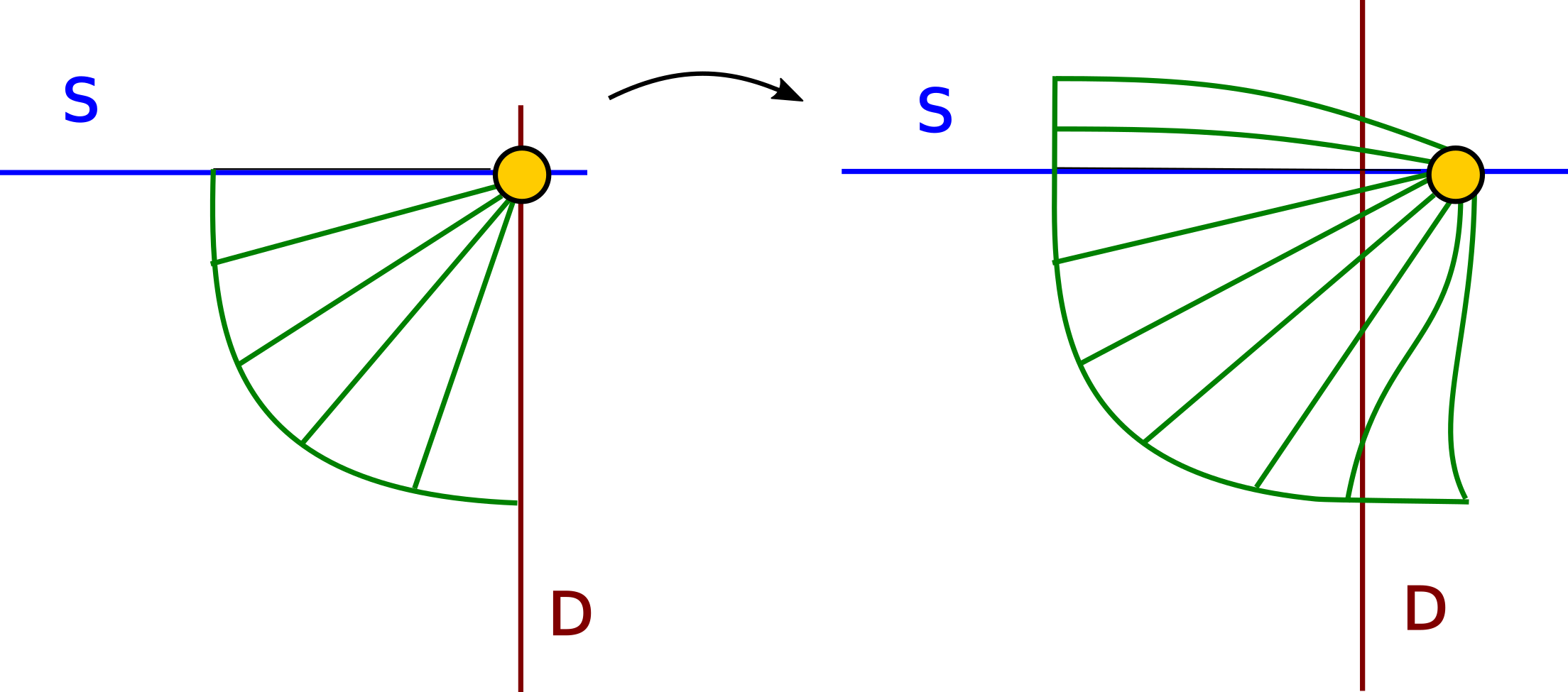}
	\caption{A depiction of $\Sigma \times \mathbb D^2\sse B$, and the modification of the Legendrian Whitney bridge $(e,F,W)$ to its extension $(\overline{e},F,W)$, which  places the extended collar $\overline{C}_\Sigma \times (-\varepsilon-1, 1+\varepsilon)$ of $\Sigma$ inside the domain $\Sigma \times \mathbb D^2$. In the picture, the angle is extended from $(0,1)$ on the right, to $(-\varepsilon-1, 1+\varepsilon)$ on the left.}   \label{fig:deformedparameter}
\end{figure}

\subsubsection{Preparation Before Sliding.} The Legendrian Whitney bridge $(e,F,W)$ provided by Proposition \ref{lem:Whitney_trick} can be intuitively depicted as an piece of a smooth open book \cite{GeigesCont} with binding $\Sigma$ and page $W$. The 2-disk factor of the tubular neighborhood $\Sigma\times D^2$ near the binding is parametrized by coordinates $(q_{n-1},q_n)$, or equivalently $(x,y)$. Before constructing the contact isotopy which removes the intersection $\Sigma=S\cap\DD$, let us perform a minor modification placing the binding of this open book off $\Sigma$, as depicted in Figure \ref{fig:deformedparameter}.

The need for this modification is already present in the smooth Whitney trick, as can be visualized in Figure \ref{fig:cornernice}. Indeed, the vector field $\dd_{q_n}$ in the above Whitney bridge $(e,F,W)$ vanishes along $\Sigma$, given that the smooth map $(F\circ e)$ only embeds the quotient $\W$. It is thus necessary to enlarge the collar $C_\Sigma$ of $\Sigma\sse \W$ in order for the extension of the vector field $\dd_{q_n}$ to displace $\D$ from $S$. This modification is stated in detail in the upcoming Lemma \ref{lem:nicecoor}.

The notation follows that of Proposition \ref{lem:Whitney_trick}, with $(s,t)$ denoted by $(q_{n-1},q_n)$ and $q\in\Sigma$ denoting a point in $\Sigma=\dd W$. Let us also prolong a collar neighborhood of the boundary $\Sigma=\dd W\sse W$ from $C_\Sigma=\Sigma \times (-\varepsilon, 0]$ to $\overline{C}_\Sigma=\Sigma \times (-\varepsilon,\varepsilon]$ and use it to define the extension $\overline{W}= W \cup \partial W \times (0, \varepsilon]$. In addition, fix any arbitrary smooth function $m: (-\varepsilon, -\varepsilon/2) \times (-1-\varepsilon, 0) \lr (-\varepsilon-1, 0)$, on the variables $(q_{n-1},q_n)$ such that:
\begin{itemize}
	\item[-] $m$ is not increasing on $q_{n-1}$,
	\item[-] For $-\varepsilon < q_{n-1} < -4\varepsilon/5$, $m(q_{n-1}, q_n)=1$, 
	\item[-] For $-3\varepsilon/5 < q_{n-1} < -\varepsilon/2$ and $q_n < -\varepsilon/10$, $m(q_{n-1}, q_n)=\varepsilon$, 
	\item[-] For $-3\varepsilon/5 < q_{n-1} < -\varepsilon/2$ and $q_n =0$, $\dd_{q_n}m=0$ and $m(q_{n-1}, q_n)=1$.
\end{itemize}

In the working hypothesis of Proposition \ref{lem:Whitney_trick}, we have the following:

\begin{lemma} \label{lem:nicecoor}
	There exists a smooth map $\overline{e}: \overline{W}  \times [-\varepsilon-1, 1+\varepsilon] \lr \W$ such that
	\begin{itemize}
		\item[(a)] The map $\overline{e}$ coincides with $e$ away from $\Op(\Sigma)\sse W$, i.e. on the domain $(\overline W \setminus \overline{C}_{\Sigma}) \times [0,1]$,
		
		\item[(b)] The  domain $\Sigma \times (-\varepsilon, \varepsilon] \times [-\varepsilon-1, 1+\varepsilon]$ is mapped by $\overline{e}$ into the domain $\Sigma \times \mathbb D^2(3 \varepsilon) \sse B$,
		
		\item[(c)] Let $\overline{\psi}=F \circ \overline{e}$, then $\overline{\psi} (p,q_{n-1},q_n)=(p, \overline{h}(q_{n-1},q_n))$ on $\Sigma \times (-\varepsilon, -\varepsilon/2) \times [-1-\varepsilon, 1 + \varepsilon]$, where the smooth map $\overline{h}$ is given by  
		
		\begin{eqnarray*}
			\overline{h}(q_{n-1},q_n) =(q_{n-1}, -q_n \cdot m(q_{n-1}, q_n)), &  (-\varepsilon, -\varepsilon /2) \times (0,1+\varepsilon),\, \,     \\
			\overline{h}(q_{n-1},q_n) = h(q_{n-1},q_n), &  (-\varepsilon, -\varepsilon /2) \times [0,1] ,\, \,   \\
			\overline{h}(q_{n-1},q_n) =(q_n-1, -q_{n-1}), &  (-\varepsilon, -\varepsilon /2) \times (-1-\varepsilon, 0),\, \, 
		\end{eqnarray*}
		
		\item[(d)]  $\overline{\psi}(q,q_{n-1},q_n)= (q, q_{n-1}+2\varepsilon,0)$ in the domain $\Sigma \times (\varepsilon/2, \varepsilon) \times [-\varepsilon-1, 1 + \varepsilon]$,
		
		\item[(e)] In the domain $\Sigma \times (-\varepsilon/2, \varepsilon /2) \times [-\varepsilon-1, 1 +\varepsilon)]$, we have that
		$$\overline{\psi}(p,q_{n-1},q_n)= (p, \overline x(q_{n-1},q_n), \overline y(q_{n-1},q_n)).$$
		
		where the functions
		$$\overline x:  (-\varepsilon/2, \varepsilon/2) \times [-\varepsilon-1, 1 +\varepsilon] \lr [-\varepsilon/2, 5\varepsilon/2],$$
		$$\overline y:  (-\varepsilon/2, \varepsilon/2) \times [-\varepsilon-1, 1 +\varepsilon] \lr [-\varepsilon/2, \varepsilon]$$
		satisfy the inequalities $\dd_{q_{n-1}}\overline x,\dd_{q_n}\overline{x}>0$ on the intersection $(\Sigma \times (-\varepsilon/2, \varepsilon/2) \times [-\varepsilon, 1 +\varepsilon]) \cap \mathbb D$ if the coordinate $\overline{x}$ itself vanishes,
		
		\item[(f)] $e(W \times \{ 0 \}) \subset \overline e (\overline W \times \{ 0 \})$ and $\overline e(\overline W \times \{ 0 \}) \cap \Sigma \times D^2(3\varepsilon)= \Sigma \times (-3\varepsilon, 3\varepsilon) \times \{ 0 \}$.
	\end{itemize}
\end{lemma}

\begin{proof}
	
	The items $(a)$-$(d)$ and $(g)$ are proven similarly as Proposition \ref{lem:Whitney_trick}. The main difference is the addition of the inequalities $\dd_{q_{n-1}}\overline x,\dd_{q_n}\overline{x}>0$ in item $(e)$. For that, just note that the condition
	$\frac{\partial  q_n}{\partial \overline x} >0$ is equivalent to 
	$\frac{\partial \overline x}{\partial  q_n} >0$. Since the initial morphism $e$ satisfies $\frac{\partial x}{\partial  q_{n-1}} \geq 0$ and $\frac{\partial x}{\partial  q_n} \geq 0$ globally, we can ensure that the extension $\overline{e}$ satisfies this property $(e)$.
\end{proof}

In order to ease notation, we shall refer to this extended Legendrian Whitney bridge still as $(e,F,W)$, understanding that the goal is to displace the original intersection $\Sigma$. However, the coordinates $(x,y)$ will be still referring to the previous embedding and this is why they satisfy Property $(e)$

\begin{figure}[h!]
	\includegraphics[scale=0.55]{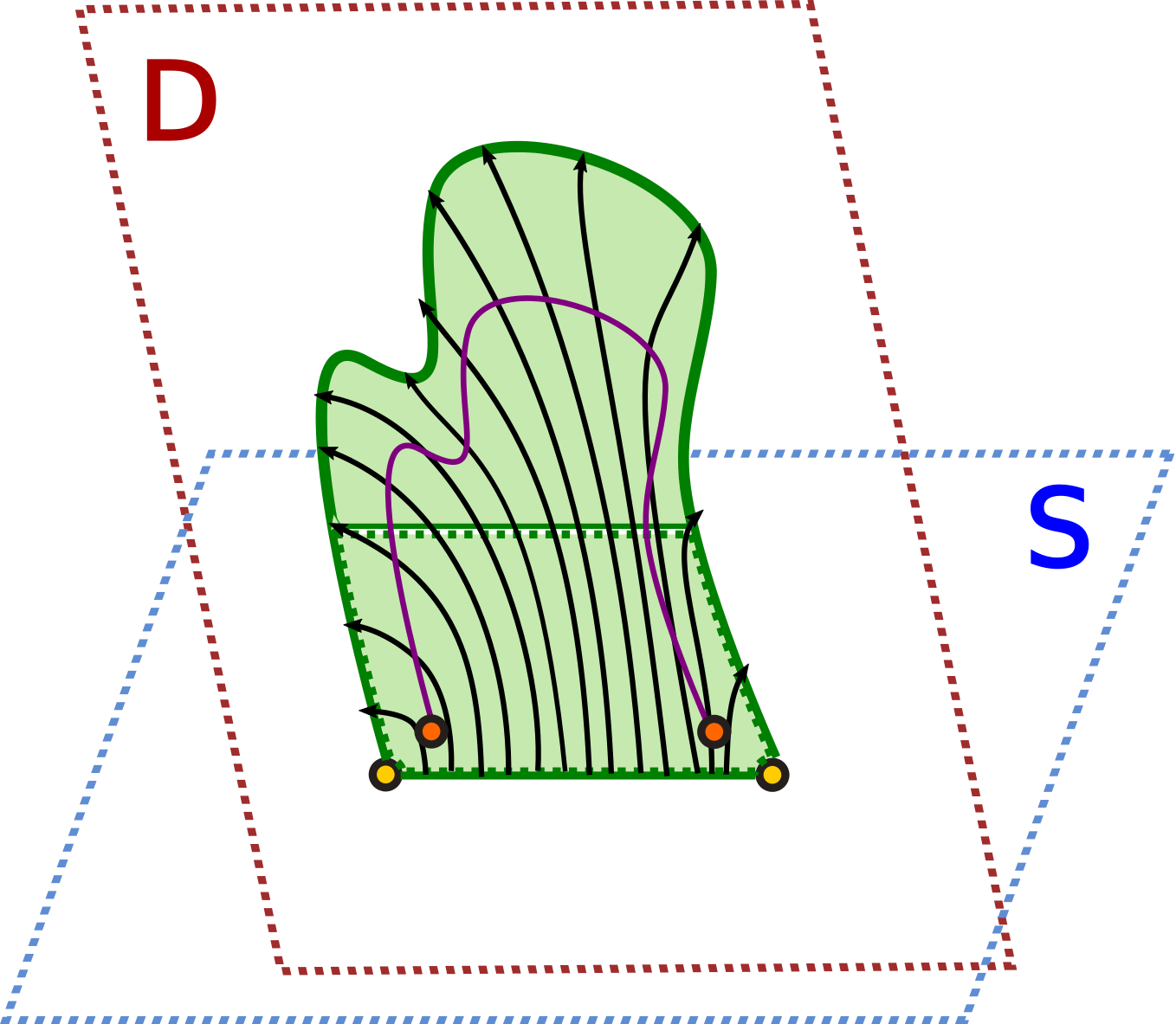}
	\caption{The Legendrian neighborhood $\Sigma \times \mathbb D^2$ of the modified intersection $\Sigma=S\cap\DD$ (Yellow) and the image of the extended Legendrian Whitney bridge (Green), where $W\times [0,1]$  is mapped into the third quadrant. The purple curve represents the original boundary of the Legendrian Whitney bridge $(e,F,W)$ and the flowlines of extended vector field $\dd_{q_n}$ (Black) displacing the original $\Sigma$ (Orange).}\label{fig:cornernice}
\end{figure}

\subsubsection{The contact isotopy} Endowed with the Legendrian Whitney bridge provided by Proposition \ref{lem:Whitney_trick}, and duly modified in Lemma \ref{lem:nicecoor}, we can consider the smooth isotopy given by the flow along the (quotient of the) coordinate $q_n\in[0,1]\sse \W$, generated by the vector field $\dd_{q_n}\in\Gamma(T\W)$, which is non-vanishing away from $\Sigma$, and adjusted to remove the intersection at $\Sigma$. This is the smooth generalization of the classical Whitney trick, as depicted in Figure \ref{fig:corner}. This smooth flow, defined along $\W$, lifts to a contact isotopy in the 1-jet bundle $J^1(\W,\xi_{st})$ \cite{GeigesCont}. Nevertheless, this is {\it not} a compactly supported contact isotopy.

The canonical lift of the vector field $\overline{e}_*(\dd_{q_n})$ has contact flow $ (j^1 \phi)^t(w, q_n, p_w, p_n, \tau)= (w, q_n+t, p_w, p_n, \tau)$, where $w\in W$ belongs to the Legendrian Whitney bridge. Note that this contact flow displaces the Legendrian $S$ from the contact submanifold $\DD$ at $t\geq 1+\varepsilon$. Indeed, the contact flow $(j^1 \phi)^t$ preserves the $p_n$ coordinate and we have the following:

\begin{lemma} \label{lem:partialpush}
Let $z\in \Sigma \times D^2(3\varepsilon)$ be such that the coordinate $p_n(z) \neq 0$ is non-vanishing, $p_{n-1}(z)=0$ and $\overline x(z)=0$. Then $p_x(z) \neq 0$ and thus $z\not \in \mathbb D$.
\end{lemma}

\begin{proof}
Let us verify $p_n \neq 0$ implies $p_x\neq 0$. Indeed, being a change of coordinates we have
	$$p_n dq_n+ p_{n-1}dq_{n-1}= p_x d\overline x + p_y d\overline y,$$
	and, by the hypothesis $p_{n-1}=0$, we have that
	$$ p_n dq_n + 0 dq_{n-1}= p_n dq_n =p_n \frac{\partial  q_n}{\partial \overline x} dx + p_n \frac{\partial q_n} {\partial \overline y} dy. $$
	Thus we have $p_x= p_n \frac{\partial  q_n}{\partial \overline x}$ and, by Lemma \ref{lem:nicecoor}, we have $\frac{\partial  q_n}{\partial \overline x}>0$. 
\end{proof}

Lemma \ref{lem:partialpush} shows that 
$$p_n((j^1 \phi)^t(w,0, 0, p_n,0))\neq 0 \Longrightarrow p_x((j^1 \phi)^t(w,0, 0, p_n,0))\neq 0.$$ 
This implies that any point of $z\equiv (w,0,0,p_n,0) \in S \setminus W_0$ satisfies that $(j^1 \phi)^t(w,0, 0, p_n,0)$ does not belong to $\DD$ for any $t>0$. This shows that the only intersection of the displaced Legendrian spheres is happening along the image of $W_0$ through the sliding flow. These intersections, in turn, disappear at time $t=1+ \varepsilon$. In addition, observe that this property remains true after applying the cut--offs which achieve the compact supported for the contact isotopy. Since the displacement occurs in the $q_n$ directions and the Legendrian Whitney bridge essentially moves along the $(q_{n-1},q_n)$, the core work for achieving a compactly supported contact isotopy must be in cutting-off along the conjugate coordinates $(p_{n-1},p_n)$ and the Reeb direction $\tau$. Let us start constructing this cut-off.

Let us start in a neighborhood of $W_0=S\cap\W$, which is one of the two ends of the Whitney bridge. Let $c\in\R^+$ be a positive real number $c$ such that $0< c <\varepsilon/2$, and construct a compactly supported cut-off function $\widetilde G: (-3\varepsilon, 3\varepsilon) \lr [0,1]$ satisfying that
\begin{itemize}
	\item[-] $\widetilde G(0)=1$, $\widetilde G'(0)=0$, and $\widetilde G'$ is odd,
	\item[-] For all $p_n\in (-3\varepsilon, 3\varepsilon)$, the associated integral $$\widetilde g(p_n)= \int_0^{p_n}  \widetilde{G}'(s)\cdot s \cdot ds$$
	satisfies $|\widetilde g|_{C_0} \leq 3c$.
\end{itemize}

\begin{figure}[ht]
	\includegraphics[scale=0.3]{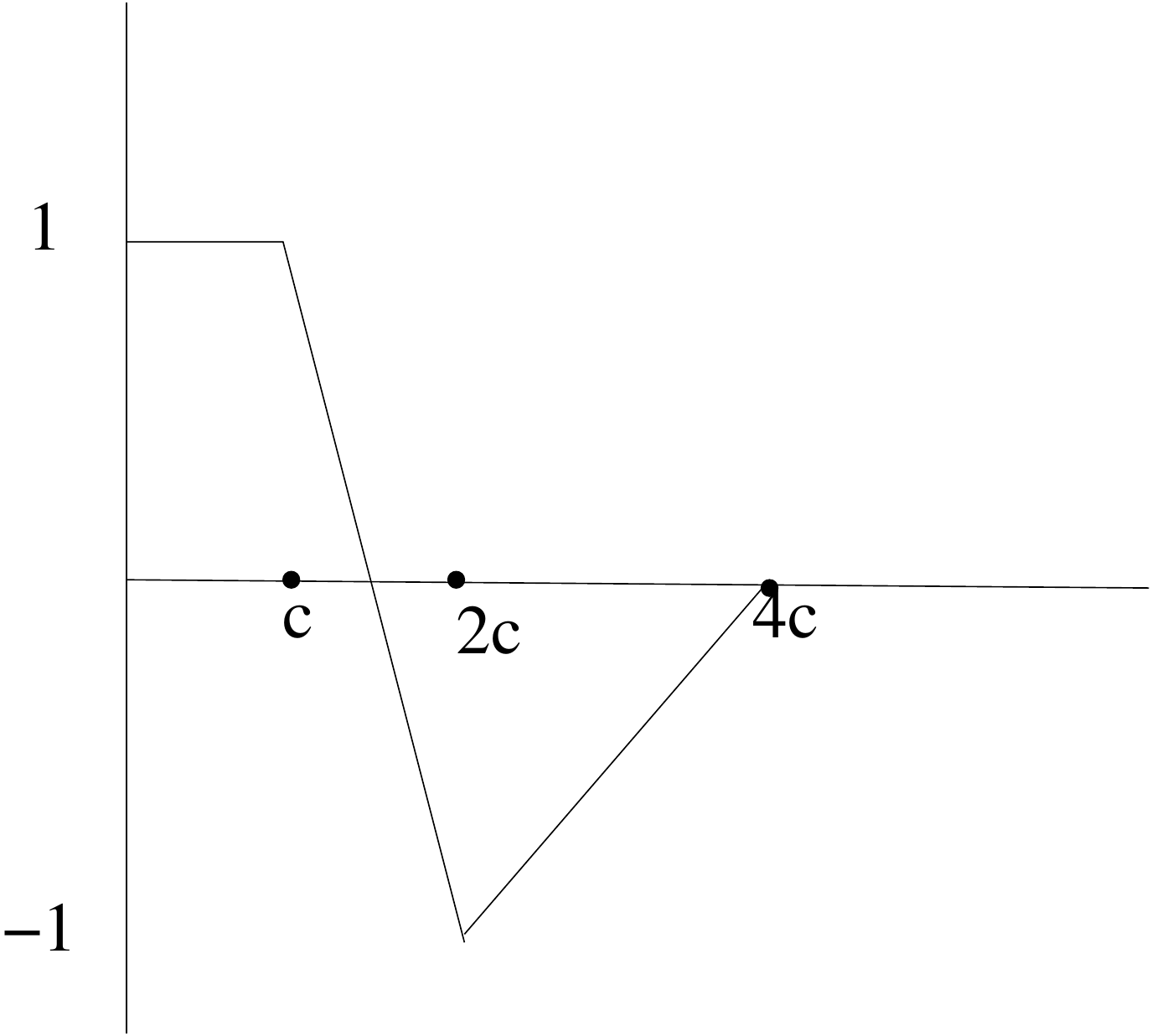}
	\caption{The figure depicts the linearized version $\widetilde{G}$ of the function $G$ }\label{fig:function_G}
\end{figure}

In order to construct such $\widetilde{G}$, we can proceed as follows. Consider the function
\begin{equation*}
\widetilde{G} = \left\{ \begin{array}{cc} 1 & p_n\in [0,c]  \\ 1-2\frac{p_n-c}{c} & p_n\in [c,2c] \\ -1 + \frac{p_n-2c}{2c} & p_n\in [2c,4c]  \\ 0 & p_n \geq 4c \end{array},\right.
\end{equation*}
which we extend to negative values by imposing even symmetry. Then $\widetilde{G}$ satisfies
\begin{equation}\int_0^{4c} \widetilde{G}'(p_n) \cdot p_n dp_n =0, \label{eq:zeroarea}
\end{equation}
Since $\widetilde{G}'$ is an odd function,  $(\widetilde{G}'(p_n) )\cdot p_n$ is even and thus the function $\widetilde{g}: (-3\varepsilon, 3\varepsilon) \lr \R$, defined as
$$ \widetilde{g}(p_n) = \int_{0}^{p_n} \widetilde{G}'(s)\cdot s \cdot ds, $$ 
is an odd function which is also compactly supported in $[-4c,4c]$ after Equation (\ref{eq:zeroarea}). Since $|\widetilde g|_{C_0} \leq 3c$, the above $\widetilde{G}$ satisfies the second property. Finally, smooth out the function $\widetilde{G}$ to a new $C^{\infty}$ function $G$, which is $C^0$-close to it, still has support in the closed interval $[-4c,4c]$ and satisfies the required properties. In particular, the associated function
\begin{equation*}
g(p_n) = \int_{-3\varepsilon}^{p_n} G'(s)s ds, \label{eq:g}
\end{equation*}
has compact support in the interval $[-4c,4c]$, and satisfies $g'(p_n)=G'(p_n)\cdot p_n$ and $|g|_{C^0} \leq 3c$. Note that $G$ can be chosen such that the norm $|g|_{C^0}$ is arbitrary small, by choosing $c$ as small as required.

Now, the flow $\Psi_t(w, q_n, p_w, p_n, \tau)=( w, q_n+tG(p_n), p_w, p_n, \tau- tg(p_n))$ is a strict contact flow defined for $t<1+\varepsilon$ on a neighborhood of the intersection $ S \cap\W$. It is here assumed that the function $g$ is $C^0$-small to assure that the map $\tau \lr \tau- tg(p_n)$ is well-defined, since the $\tau$-coordinate is defined for a small open neighborhood of $\tau=0$. In the same vein, the inequality $|G|_{C^0} \leq 1$ implies that the function $q_n+tG(p_n)$ lies in the interval $(-\varepsilon-1, 1 +\varepsilon)$.

In addition, this cut-off flow is the identity for $p_n> 3\varepsilon$ and therefore it glues $\Psi_t(S)$ and $S$ in the transition region $\Op(\{ p_n = \pm 3 \varepsilon \})$. Thus, this discussion above resolves the contact cut-off regarding the direction associated to the $1$-jet of the coordinate $q_n$. Let us now study the remaining region $q_{n-1}> \varepsilon$, responsible for the {\it horizontal} directions in the Whitney trick, as depicted in Figure \ref{fig:corner}.

The  contact vector field associated to the above constructed flow is given by $X_H= G(p_n)\frac{\partial}{\partial q_n} -g(p_n) \frac{\partial}{\partial \tau}$, whose associated Hamiltonian function is $H= G(p_n) \cdot p_n -g(p_n)$. Choose a real number  $\delta \in (-\varepsilon, \varepsilon)$ such that for any point $z \in\Op(\W)$ with coordinate $q_{n-1}(z) \geq \delta$, we have $x(z) >0$. Fix a smooth step function $R: (-\infty, \varepsilon) \lr [0,1] $ such that
\begin{itemize}
	\item[-] $R(q_{n-1})= 1$ for $q_{n-1} \leq \frac{\varepsilon+\delta}{2}$,
	\item[-] $R(q_{n-1})= 0$ for $q_{n-1} \geq \frac{3\varepsilon+\delta}{4}$,
	\item[-] $0 \geq \dd_{q_{n-1}}R \geq -\frac{5}{\varepsilon-\delta}$.
\end{itemize}
and define $\hat R =R \cdot H$, where $H$ is the Hamiltonian above. Then the contact vector field associated to the Hamiltonian $\hat R$ is given by the formula $X_{\hat H}= R \cdot X_H - (\dd_{q_{n-1}}R) \cdot H\cdot \dd_{p_{n-1}}$. Thus the associated contact flow $\hat{\Psi}_t$ coincides with $\Psi_t$ away from $\mathcal O p(\Sigma \times D^2(3\varepsilon))$. In the neighborhood $\mathcal O p(\Sigma \times D^2(3\varepsilon))$ we obtain
$$ \hat{\Psi}_t(q, q_{n-1}, q_n, p, p_{n-1}, p_n, \tau)=( q, q_{n-1}, q_n+tG(p_n)R(q_{n-1}), p, p_{n-1}+ t\cdot \dd_{q_{n-1}}R \cdot H, p_n, \tau- tg(p_n)R(q_{n-1})).$$
Since the upper bound $\|\dd_{q_{n-1}}R \cdot H\| \leq \frac{5}{\varepsilon-\delta} \cdot (c+ 2c)= \frac{15c}{\varepsilon-\delta}$ can be made arbitrary small by further reducing $c>0$, the flow of $\hat{\Psi}_t$ is well-defined for time $t<1+ \varepsilon$. In consequence, we obtain a $1$-parametric family of Legendrian submanifolds obtained by flowing the Legendrian $S$ along the contact isotopy $\hat{\Psi}_t$, for $t\in [0,1+\varepsilon)$. In addition, the Legendrian $S_{1+\varepsilon}$ is disjoint from the contact submanifold $\mathbb D$. Indeed, in the region $S_{1+\varepsilon} \cap (B \setminus \Op(\W))$, the family of Legendrians $S_t$ remains constant and thus the empty intersection with $\mathbb D$ persists. In the region $S\cap\Op(W)$ near the Legendrian Whitney bridge, the coordinate $p_n$  is preserved along the contact flow. It thus suffices to show that the Legendrian $S_t$ can only intersect $\mathbb D$ in points of the image of the domain $W_0= S \cap \{ p_n=0 \}$ through the flow. This can be discussed in the following two regions:

\begin{itemize}
	\item[(a)] The complement of the region $\Sigma \times \DD(3\varepsilon)$. Then $p_x=p_n$ for $x=0$ and thus since $p_n$ is preserved, we obtain that the only possible intersections between $\DD$ and $\hat{\Psi}_t(S)$ happen in the image of $W_0$. The same argument works for the region in $\Sigma \times \DD(3\varepsilon)$ in the complement of the support of $R$.
	
	\item[(b)] In the remaining region, where the coordinate $p_{n-1}$ is not preserved, we cannot use Lemma \ref{lem:partialpush}. Nevertheless, the inequality $x>0$ is satisfied and preserved by the contact flow in positive time. Thus, since $\DD$ is defined by $\{ x=p_x=0 \}$, the image of $S$ can only intersect $\DD$ along the image of $W_0$ through the flow.
\end{itemize}

In consequence, since the coordinate $p_n$ is preserved by the flow, the contact flow coincides with the displacement flow constructed along $\W$, and thus at time $1+\varepsilon$ the intersection between $S_{1+\varepsilon}$ and $\DD$ must be empty.\hfill $\Box$


\section{Application to Isocontact embeddings} \label{sec:main}

In this section we establish Theorem~\ref{thm:main} by using Theorem \ref{thm:Legenvoid}. In line with standard h-principle arguments \cite{CPPP,CP,EliashMisch,Gr86}, the isocontact h-principle Theorem \ref{thm:main} follows from the resolution of the local problem, which is the content of the following:

\begin{theorem}\label{thm:iso-contact_disk_embeddding_rel_boundary}
	Let $(B^{2n+1}, \xi)$ be a contact $(2n+1)$--dimensional disk, $n\geq2$, and $(\phi, F_s):(\mathbb D^{2n-1}, \xi_{std}) \lr(B^{2n+1}, \xi)$
	a proper formal isocontact embedding such that:
	
	\begin{enumerate}
		\item The embedding is isocontact close to the boundary, i.e $(F_s)_{|\SO p(\mathbb D^{2n+1})}= d\phi_{|\SO p(\mathbb D^{2n+1})}$,
		\item The embedding is smoothly isotopic to the standard embedding.
	\end{enumerate}
	Then, there exists a family of formally isocontact 
	embeddings $(\phi^t, F_s^t)_{t\in [0,1]}$, relative to the boundary, such that the deformed embedding is isocontact; i.e. $F_s^1=d\phi^1$.\hfill$\Box$
\end{theorem}

Theorem \ref{thm:iso-contact_disk_embeddding_rel_boundary} constructs a solution to the core local model in our existence h-principle. In comparison to the recent work \cite{HondaHuang19}, it is crucial to emphasize that the contact embedding $\phi^1$ induces the standard contact structure $(\mathbb D^{2n-1}, \xi_{std})$, which is the difference between the study of contact embeddings and the existence h-principle for isocontact embeddings.

Theorem~\ref{thm:main} follows from Theorem~\ref{thm:iso-contact_disk_embeddding_rel_boundary}, thanks to the h-principle for isocontact embeddings between open contact manifolds \cite{EliashMisch,Gr86} applied to the formal isocontact embedding
$$(f_0, F_s)|_{N \setminus \overline{B_N}}:
N \setminus \overline{B_N} \lr M,$$
where $B_N$ is a Darboux ball in $(N,\xi_N)$. Indeed \cite[Theorem 12.3.1]{EliashMisch} provides a formal isocontact deformation of $(f_0, F_s)|_{N \setminus \overline{B_N}}$ to a genuinely isocontact embedding, and thus to the local model for Theorem \ref{thm:iso-contact_disk_embeddding_rel_boundary}.

\subsection{Intuition for Theorem \ref{thm:iso-contact_disk_embeddding_rel_boundary}}. Let us first consider the 3-dimensional case of a smooth knot $K\sse (S^3,\xi_{st})$ or, in the local context of Theorem \ref{thm:iso-contact_disk_embeddding_rel_boundary}, a properly embedded smooth arc $K_0\sse (B^3,\xi_{st})$ within a Darboux ball $(B^3,\xi_{st})$ which is positively transverse at the endpoints $\dd K_0$. The aim is to perform a smooth isotopy $K_t$, $t\in[0,1]$, relative to the boundary $\Op(B^3)$ such that $K_1$ is a transverse arc. In the 3-dimensional case this can be done explicitly, as explained in \cite[Theorem 3.3.1]{GeigesCont}. In short, the case where the transverse knot is positively transverse everywhere except at Legendrian subsets can be solved directly by averaging the positivity at the ends. Thus, the only significant reason for $K_0$ not to be a transverse knot is the existence of interval subsets $I\sse K_0$ where the arc is negatively transverse. The solution is to insert a 1-dimensional wrinkle \cite{EMWrinkle}, i.e. the transverse push-off of a Legendrian stabilization. In conclusion, in the 3-dimensional case the problem translates into a solvable question on plane curves and their slopes, an approach which does not exist in higher dimensions.

The crucial object in the 3-dimensional case is the notion of a transverse stabilization, which is neatly explained in \cite[Section 2.8]{Etnyre}. As depicted in \cite[Figure 21]{Etnyre}, the transverse stabilization can be understood as a {\it transverse} bypass. This is where intuition for the higher-dimensional case starts to arise. A Legendrian bypass \cite{HondaBypass} can be understood as half an overtwisted 2-disk and, even more precisely, the connected sum of a Legendrian arc with the boundary of an overtwisted disk is tantamount to a Legendrian stabilization. Thus, the overtwisted disk can be understood as an object which provides the necessary stabilizations required for the smooth arc to become a transverse (or Legendrian) arc. In fact, and this is crucial, only half of the overtwisted 2-disk is needed.

The coming proof of Theorem \ref{thm:iso-contact_disk_embeddding_rel_boundary} is based on the following 3-dimensional geometric heuristic. Instead of searching for half an overtwisted disk in the Darboux ball $(B^3,\xi_{st})$, which allows for transverse stabilizations, insert an entire overtwisted 2-disk $\D$ away from $\Op(K_0)$. This modifies the standard structure $(B^3,\xi_{st})$ to a overtwisted structure $(B^3,\xi_{ot})$, in which the h-principle for isocontact embeddings holds \cite{BEM,El}. Thus, the smooth inclusion $K_0\sse (B^3,\xi_{ot})$ can be made positively transverse, relative to the endpoints. By Gray's Isotopy Theorem, this can be considered as a smooth isotopy $\{K_t\}_{t\in[0,1]}$ within the fixed contact structure $\xi_{ot}$, such that $K_1$ is positively transverse. The h-principle does not a priori guarantee that the isotopy $\{K_t\}_{t\in[0,1]}$ only uses half the overtwisted 2-disk $\D$. (Although, in this 3-dimensional case, we know we can just use half overtwisted 2-disk, but such statement is not available in higher dimensions.) The argument then relies on proving that any such (formally transverse) isotopy $\{K_t\}_{t\in[0,1]}$ can be made disjoint from half of the overtwisted 2-disk $\D$ and that the isotopy $\{K_t\}_{t\in[0,1]}$ can therefore be pushed to exist within $(B^3,\xi_{st})$, where only half of the overtwisted 2-disk is available.

In the higher-dimensional case, the generalization of the arc $K_0\sse (B^3,\xi_\st)$ is a standardly embedded $(2n-1)$-dimensional disk $(\DD,\xi_\st)\sse (B^{2n+1},\xi_\st)$. The fact that Theorem \ref{thm:iso-contact_disk_embeddding_rel_boundary} is able to conclude an embedding of the standard contact structure $(\DD,\xi_\st)$ is a feature of both Theorem \ref{thm:Legenvoid} and the h-principle \cite{BEM}. In the 3-dimensional case, there is a unique contact structure induced in a transverse arc, whereas a codimension-2 contact embedding of a smooth disk $\DD$ does not necessarily induce the standard contact structure $(\DD,\xi_\st)$. In addition, the formal contact type of a smooth embedding in higher dimensions differs from the data of the self-linking number of a transverse knot, which entirely determines the formal transverse type of a smooth knot in $(S^3,\xi_\st)$. Finally, the proof of Theorem \ref{thm:iso-contact_disk_embeddding_rel_boundary} only works in higher dimensions, since it relies on the smooth Whitney trick \cite{Kir,Whitney44}, its Legendrian avatar Theorem \ref{thm:Legenvoid}, and the fact that the space of embedded $n$-sphere into $\R^{2n+1}$ is connected, which only hold for $n\geq2$. The following subsection comprises the details of our argument for Theorem \ref{thm:iso-contact_disk_embeddding_rel_boundary} in higher dimensions.

\subsection{Proof of Theorem~\ref{thm:iso-contact_disk_embeddding_rel_boundary}}. First, we shall modify the contact structure $(B,\xi)$, away from the image $\phi(\DD)$, by inserting an overtwisted disk. It is possible to achieve this by choosing a standard Legendrian unknot $\La_0\sse(B,\xi)$, and performing\footnote{Alternatively, choose a stabilized Legendrian unknot $\La \sse(B\setminus\DD,\xi)$, formally isotopic to $\La_0$, and perform one $(+1)$-surgery.} two $(+1)$-surgeries along $\La_0$. Let us do so in the framework of adapted contact open books \cite{CMP,Colin}. For that, choose a $(2n+1)$-dimensional Darboux ball $(\BB,\xi_\st)\sse(B\setminus\DD,\xi)$ and $p\in(\BB,\xi_\st)$ an interior point. This Darboux ball $(\BB,\xi_\st)$ is contactomorphic, relative to its boundary, to the contact connected sum $(\BB,\xi_\st)\#_p(S^{2n+1},\xi_\st)$ along (a neighborhood of) the point $p\in(\BB,\xi_\st)$. The contact structure $(S^{2n+1},\xi_\st)$ is supported by the adapted open book $(S^{2n+1},\xi_\st)=\ob(T^*S^n,\la_\st;\tau_{S^n})$, and the contact connected sum can be performed such that $p$ belongs to the zero section of the Weinstein page $P_0\cong (T^*S^n,\la_\st)$ at angle $0\in S^1$. In order to insert an overtwisted disk in this context, change the monodromy $\tau_{S^n}\in \Symp^c(T^*S^n,\la_\st)$ to its inverse $\tau^{-1}_{S^n}$, to obtain an overtwisted contact structure $(S^{2n+1},\xi_\ot)=\ob(T^*S^n,\la_\st;\tau^{-1}_{S^n})$, as we proved in \cite{CMP}. The connected sum $(\BB,\xi_\st)\#_p(S^{2n+1},\xi_\ot)$ is thus overtwisted and differs from $(\BB,\xi_\st)\#_p(S^{2n+1},\xi_\st)$ precisely in the gluing prescribed by the monodromy. Since the action of the monodromy can be concentrated in a neighborhood of the page $P_\pi$ at angle $\pi\in S^1$, these two contact connected sums can be assumed to be contactomorphic away from a neighborhood $\Op(P_\pi)$.

The formal isocontact embedding $(\phi,F_s):(\DD,\xi_\st)\lr(B,\xi)$ can be extended to have the modified overtwisted contact manifold $(B,\xi)\#(S^{2n+1},\xi_\ot)$ as a target, since its image lies away from $p$. Let us apply the h-principle \cite{BEM} for codimension-2 isocontact embeddings in overtwisted contact manifolds in order to obtain a family $\{(\phi^t,F^t_s)\}_{t\in[0,1]}$ of formally isocontact embeddings which starts at $(\phi,F_s)$ and finishes at a genuinely isocontact embedding $(\phi^1,F^1_s)$. Even though the image $(\phi^0,F^0_s)=(\phi,F_s)$ lies away from any overtwisted disk, the image of the embeddings $(\phi^t,F^t_s)$ might, for some $t\in[0,1]$, might be contained in the $(\BB,\xi_\st)\#_p(S^{2n+1},\xi_\ot)$ region near $p$. The crucial fact is that, in order to avoid the region $(\BB,\xi_\st)\#_p(S^{2n+1},\xi_\ot)$, it suffices to avoid the Legendrian sphere $S_\pi\sse P_\pi\cong(T^*S^n,\la_\st)$ given by the Legendrian lift of the (exact) Lagrangian zero section. This is explained in the following:

\begin{lemma} \label{lemma:flowing}
	Let $\varepsilon\in\R^+$, $\theta\in(\pi-\varepsilon, \pi +\varepsilon)$ and $K\sse N=(\mathbb{D}_{\varepsilon}(T^*S^n) \times (\pi-\varepsilon, \pi +\varepsilon),\ker\{\la_\st+d\theta\})$ be a compact subset of $N$ such that $K\cap (S^n\times\{\pi\}) =\emptyset$, where $S^n\sse T^*S^n$ is the inclusion given by the zero section. Then, there exists a compactly supported contact flow $\p_t\in\Cont(N)$ such that $\p_1(K)\sse\Op(\dd N)$.
\end{lemma}

\begin{proof}
The complement of the zero section $S^n\sse(\mathbb{D}_{\varepsilon}(T^*S^n),\la_\st)$ is Liouville symplectomorphic to the symplectization $(\S(T^*S^n) \times \R,e^t\alpha_{std})$, $t\in\R$, of the standard unit cotangent bundle $(\S(T^*S^n),\ker\alpha_\st)$. Thus, the complement of the zero section is modeled on the contactization of the symplectization of $(\S(T^*S^n),\ker\alpha_\st)$, with contact form $e^t\alpha_{std} +d\theta$. The vector field $X= \partial_t + \partial_\theta$ is contact and a pseudo--gradient for the function $t+ \theta$. Choose a compactly supported cut-off function $\hat{H}$ that coincides with its contact Hamiltonian $H$ on a compact set of $N$ whose complement is a small open negihborhood of $\partial N$. The flow associated to the contact vector field generated by the contact Hamiltonian $\hat{H}$ satisfies the requirements, as it pushes the complement of the zero section to $\Op(\dd N)$.
\end{proof}

Lemma \ref{lemma:flowing} shows that an embedding whose image lies on the complement of the Legendrian $S_\pi$ page inside $(\BB,\xi_\st)\#_p(S^{2n+1},\xi_\ot)$ can be pushed-off away, with a contact isotopy, from the region where the contact domains $(\BB,\xi_\st)\#_p(S^{2n+1},\xi_\ot)$ and $(\BB,\xi_\st)\#_p(S^{2n+1},\xi_\st)$ differ. Since $S_\pi$ is a standardly embedded Legendrian sphere in a ball $(B,\xi)$, the Legendrian Whitney Trick, stated in Theorem \ref{thm:Legenvoid} above, can now be used to remove the intersection between the contact disks given by the image of the isocontact embedding $(\phi^1,F^1_s)$ and the Legendrian sphere $S_\pi$. At this stage, the proof could proceed by establishing a parametric version of Theorem \ref{thm:Legenvoid} which would allow for the removal of intersections between the Legendrian sphere $S_\pi$ and $(\phi^t,F^t_s)$ for all $t\in[0,1]$. Instead, given that the maps $(\phi^t,F^t_s)$ are just formal isocontact embeddings, it suffices to treat this as a differential topology problem, rather than a contact topological problem, as in the following

\begin{lemma}\label{lem:disj}
	Let $\phi_t: \DD^{2n-1} \lr B^{2n+1}$, $t\in[0,1]$, be a family of smoothly standard proper embeddings of the disk such that $n\geq2$, $\phi_1$ does not intersect  $S_\pi  \sse B$, a standard $n$-sphere in the ball, and $\phi_t(\Op(\DD^{2n-1}))$ is independent of $t\in[0,1]$. Then, there exists a deformation through paths of embeddings $\phi_{t,s}$ to a new $1$--parametric family of embeddings $\phi_{t,1}: \DD \lr B$ such that:
	\begin{itemize}
		\item[-] $\phi_{t,0}= \phi_t$, $\phi_{t,s}= \phi_0$ in $\Op(\dd(\DD^{2n-1}))$, for all $(t,s)\in[0,1]\times[0,1]$,
		\item[-] For all $s\in [0,1]$, $\phi_{0,s}= \phi_0$ and $\phi_{1,s}= \phi_1$, 
		\item[-] The images of $\phi_{t,1}$ do not intersect $S_{\pi}$, for all $t\in[0,1]$.
	\end{itemize}
\end{lemma}

\begin{proof} The diffeomorphism $B^{2n+1}\cong \DD^{2n-1}\times \DD^2$ implies that $B\setminus\phi_1(\DD)$ is diffeomorphic to $S^1\times B^{2n}$. Proposition \ref{prop:connect}, proven below, shows that the space of smoothly embedded $n$-spheres in $B^{2n+1}\cong \DD^{2n-1}$ is connected. By deforming $S_\pi$ to a sufficiently small sphere with (the endpoint of) an isotopy $\Psi_1$, the composition of $\Psi_1^{-1}\circ\phi_t$ is our required deformation.
\end{proof}

In conclusion, the initial formal isocontact embedding is made genuinely isocontact by inserting an overtwisted disk and using the $h$-principle \cite{BEM}. Theorem \ref{thm:Legenvoid} is used to ensure that the resulting isocontact embedding avoids the Legendrian sphere $S_\pi$, i.e. that it avoids part of the overtwisted disk. Lemma \ref{lemma:flowing} is then used to flow the image of the isocontact embedding away from the region where the overtwisted disk is inserted, and Lemma \ref{lem:disj} ensures that the formal isocontact isotopy is also disjoint from $S_\pi$. This yields the required isocontact embedding for Theorem \ref{thm:iso-contact_disk_embeddding_rel_boundary}.\hfill$\Box$


Let us conclude the article with Proposition \ref{prop:connect}, on the connectedness of the space of embeddings into $\S^1 \times B^{2n}$. This is the higher-dimensional analogue of the $3$-dimensional arguments in \cite{hatcher} on the topology of the moduli space of unknots. The argument strongly relies on the connectedness of the space of embeddings into $B^{2n+1}$ case, as proven by A. Haefliger in \cite[Theorem 1]{Ha0}.

\begin{proposition} \label{prop:connect}
The space of standard embeddings of the $n$-sphere into $\S^1 \times B^{2n}$ is connected, for $n\geq 2$.
\end{proposition}

\begin{proof}

Let us denote $X=\S^1 \times B^{2n}$, and $\pi: X \lr B^{2n}$ be the projection into the second factor. Let $e:S \lr X$ be a standard embedding, with components $e=(\theta, m)$, where $\theta$ denotes the angle $\theta\in\S^1$ and assume, after genericity and possibly adding a finite number of kinks, that the smooth map $\pi\circ e:S\lr B^{2n}$ is a generic immersion and the self-intersection points are all paired, i.e. the signed self-intersection number is zero. Here, a kink is a smooth isotopy of the embedding $e:S\lr X$ which adds exactly a double point on its image $\pi(e(S))$ of the required sign. 

Let us now choose a pair of canceling points $q_0,q_1\in B^{2n}$ for $\pi(e(S))$ and a Whitney disk $d:\D^2\lr B^{2n}$ for the smooth map $\pi\circ e:S\lr B^{2n}$. The sliding, through a Whitney-move isotopy \cite{Kir}, along this Whitney disk yields a homotopy through immersions $L_t$ of the immersion $L_0= \pi(e(S))$ into an embedded sphere $L_1$. This homotopy lies entirely in $B^{2n}$, and we need to choose a lift for the family of maps $m_t: S \lr B$, defined by $L_t=m_t(S)$, to a family of embeddings $e_t: S \lr X$. For instance, choose a lift $\widetilde d:\D^2\lr \S^1\times B^{2n}$ for the Whitney disk $d:\D^2\lr B^{2n}$ such that the interior of $\widetilde d(\D^2)$ does not intersect $S$, the lower hemisphere of $\dd(\widetilde d(\D^2))$ still intersects $e(S)$ and the upper hemisphere of $\dd(\widetilde d(\D^2))$ does not intersect $e(S)$.

The Whitney disk $\widetilde d:\D^2\lr \S^1\times B^{2n}$ can now be used to perform a Whitney-move isotopy along it such that both images $e(S)$ and $\pi(e(S))$ through this smooth isotopy remain embedded. Finally, given that the degree of the map $S\lr\S^1$ must be zero, we can isotope the smooth sphere $S$ to live within the region $\Op(p)\times B^{2n}$, where $p\in\S^1$ is a point. This yields an embedding of the smooth $n$-sphere $S$ inside $B^{2n+1}$. Then, \cite[Theorem 1]{Ha0} shows that - for $n\geq2$ - the space of standard embeddings of an $n$-sphere into $B^{2n+1}$ is connected, concluding the required result.
\end{proof}


\begin{thebibliography}{9999}

	\bibitem{Arnold} V. I. Arnold, S. M. Gusein-Zade, A. N. Varchenko. Singularities of Differentiable Maps. Volume I. Birkh\"{a}user, Boston, Basel, Sttutgart. 1985.
	
	\bibitem{Bennequin} D. Bennequin. \emph{Entrelacements et \'{e}quations de Pfaff.} Third Schnepfenried geometry conference, Vol. 1 (Schnepfenried, 1982), 87--161, 
	Ast\'{e}risque, 107-108, Soc. Math. France, Paris, 1983. 


\bibitem{BEM} M.S.~Borman, Y.~Eliashberg, E.~Murphy, \textit{Existence and classification of overtwisted contact structures in all dimensions}, to appear in Acta Mathematica.


\bibitem{CasalsEtnyre} R. Casals, J.B. Etnyre, Non-simplicity of isocontact embeddings in all higher dimensions, arXiv:1811.05455.

\bibitem{CPPP} R. Casals, J.L. P\'erez, A. del Pino, F. Presas, Existence h-principle for Engel structures, Invent. Math. Volume 210, Issue 2, (2017) 417-451.

\bibitem{CP} R. Casals, A. del Pino, Classification of Engel Knots, Math. Ann. Issue 1, Volume 371 (2018), 391-404.


\bibitem{CMP} R.~Casals, E.~Murphy, F.~Presas, \textit{Geometric criteria for overtwistedness}, J. Amer. Math. Soc. 32 (2019), no. 2, 563--604.



\bibitem{Casson} A. Casson, "Three lectures on new infinite constructions in 4-dimensional manifolds" , A la Recherche de la Topologie Perdue , Progress in Mathematics , 62 , Birkh\"auser (1986) pp. 201-244.
	
	
	\bibitem{Chekanov} Y. Chekanov. \emph{Differential algebra of Legendrian links.} Invent. Math. 150 (2002), no. 3, 441--483.
	
	\bibitem{Colin} V. Colin, Livres ouverts en g\'eom\'etrie de contact, Ast\'erisque, S\'eminaire Bourbaki 59\`eme ann\'ee, 2006-2007, 969, 91-118.
	
	
	\bibitem{El} Ya. Eliashberg, { \em Classification of overtwisted contact structures on $3$-manifolds,}
Invent. Math. 98 (1989), 623--637.
	
	
	
	\bibitem{EliashMisch} Y. Eliashberg, N. Mishachev, Introduction to the $h$--Principle. Grad. Studies in Math. 48. AMS, Providence 2002.
	
	\bibitem{EMWrinkle} Y. Eliahsberg, N.  Mishachev, Wrinkling of smooth mappings and its applications, Inven. Math. 130 (1997), 345-369. 
	
	\bibitem{Etnyre} J. Etnyre. \emph{Legendrian and transversal knots.} Handbook of knot theory, 105--185, Elsevier B. V., Amsterdam, 2005.
	
	
	\bibitem{EF} J. Etnyre, R. Furukawa, 
\emph{Braided embeddings of contact 3-manifolds in the standard contact 5-sphere.}  J. Topol. 10 (2017), no. 2, 412---446.

	\bibitem{EL} J. Etnyre, Y. Lekili, Embedding all contact $3$--manifolds in a fixed contact 5--manifold, to appear in Journal of the LMS.
	
	\bibitem{Freedman} M. Freedman, "The topology of four-dimensional manifolds", J. Diff. Geom. , 17 (1982) pp. 357-453.
	
	\bibitem{FuchsTabachnikov} D. Fuchs, S. Tabachnikov. \emph{Invariants of Legendrian and transverse knots in the standard contact space.} Topology 36 (1997), no. 5, 1025--1053.
	
	
	\bibitem{GeigesCont} H. Geiges. An Introduction to Contact Topology. Cambr. Studies in Adv. Math. 109. Cambr. Univ. Press 2008.
	
	
	
	
\bibitem{Go} Goodwillie, T. G. \emph{A multiple disjunction lemma for smooth concordance embeddings.} Mem. Amer. Math. Soc. 86 (1990), no. 431.
	
	
	
	

   \bibitem{Gr86} M. Gromov, \emph{Partial differential relations.} Ergebnisse der Mathematik und ihrer Grenzgebiete (3) [Results in Mathematics and Related Areas (3)], Vol. 9,  Springer-Verlag, Berlin.

\bibitem{Ha0} Haefliger, Andr\'e, {\emph Differentiable imbeddings. } Bull. Amer. Math. Soc. 67 1961 109--112.

\bibitem{Ha} Haefliger, Andr\'e, \emph{Knotted ($4k-1$)--spheres in $6k$-space}. Ann. of Math. (2) 75 1962 452--466. 
	
\bibitem{HatcherBook} A. Hatcher.
	\emph{Algebraic topology.} Cambridge University Press, Cambridge, 2002.
	
\bibitem{hatcher} A. Hatcher, \emph{Topological moduli spaces of knots}, available at A. Hatcher's website (see also arXiv:math/9909095).
	
	\bibitem{HondaBypass} K. Honda, On the classification of tight contact structures. I. Geom. Topol. 4:309-368, 200.
	
	\bibitem{HondaHuang19} K.~Honda, Y.~Huang, Convex hypersurface theory in contact topology, arXiv:1907.06025.
	
	
	\bibitem{Kir} R. Kirby. \emph{The Whitney trick}, Celebratio Mathematica.
	
	\bibitem{Ol} O. Lazarev \emph{Maximal contact and symplectic structures}, arXiv:1810.11728.
	
\bibitem{MW} Mabillard, Isaac; Wagner, Uli; \emph{Eliminating Tverberg points, I. An analogue of the Whitney trick}. Computational geometry (SoCG'14), 171--180, ACM, New York, 2014. 
	
	
	\bibitem{MilnorHCob} J. Milnor, Lectures on the h-cobordism Theorem, Princeton University Press (1965), Princeton.
	
	
	
	\bibitem{PP} D. Pancholi, S. Pandit, \emph{Iso--contact embeddings of manifolds in co--dimension 2.} arXiv:1808.04059
	
	\bibitem{Sha} A. Shapiro, Obstructions to the Imbedding of a Complex in a Euclidean Space.: I. The First Obstruction, Annals of Mathematics
	Second Series, Vol. 66, No. 2 (Sep., 1957), pp. 256-269.
	
\bibitem{Smale} S. Smale, On the structure of manifolds, Amer. J. Math.84 (1962), 387-399.
	

C. R. Acad. Sci. Paris 232, (1951). 142--144.
	
\bibitem{Vo} T. Vogel, \emph{Non-loose unknots, overtwisted discs, and the contact mapping class group of $\S^3$}, Geom. Funct. Anal. 28 (2018), no. 1, 228--288
		
	
	\bibitem{Whitney36} H. Whitney, Differentiable manifolds. Ann. of Math. (2) 37 (1936), no. 3, 645-680.
	
	\bibitem{Whitney44} H. Whitney, The self-intersections of a smooth n-manifold in 2n-space, Ann. Math. (2) 45 (April 1944), pp. 220-246. 	
		
\end{thebibliography}
\end{document}